\documentclass{amsart}
\usepackage{amsmath,amssymb,amscd,latexsym}

\usepackage[all]{xy}

\usepackage[usenames, dvipsnames]{color}

\newcommand{\bi}{\begin{itemize}}
\newcommand{\ei}{\end{itemize}}
\newcommand{\bn}{\begin{enumerate}}
\newcommand{\en}{\end{enumerate}}

\newcommand{\bq}{\begin{equation}}
\newcommand{\eq}{\end{equation}}

\newcommand{\A}{{\mathbb{A}}}
\newcommand{\C}{{\mathbb{C}}}
\newcommand{\DD}{\Delta}

\newcommand{\Elog}{E_{\log}}

\newcommand{\MUlog}{MU_{\log}}

\newcommand{\Q}{{\mathbb{Q}}}
\newcommand{\R}{{\mathbb{R}}}
\newcommand{\Z}{{\mathbb{Z}}}
\newcommand{\an}{\mathrm{an}}
\newcommand{\cl}{\mathrm{cl}}

\newcommand{\Hdg}{\mathrm{Hdg}}

\newcommand{\HdgMU}{\mathrm{Hdg}_{MU}}

\newcommand{\loc}{\mathrm{loc}}

\newcommand{\proj}{\mathrm{proj}}

\newcommand{\Sing}{\mathrm{Sing}_*}
\newcommand{\topp}{\mathrm{top}}

\newcommand{\Omtop}{\Omega_{\topp}}

\newcommand{\Spec}{\mathrm{Spec}\,}

\newcommand{\colim}{\operatorname*{colim}}
\newcommand{\hocolim}{\operatorname*{hocolim}}
\newcommand{\holim}{\operatorname*{holim}}
\newcommand{\Hom}{\mathrm{Hom}}
\newcommand{\Imm}{\mathrm{Im}\,}

\newcommand{\Sm}{\mathbf{Sm}}
\newcommand{\Smc}{\Sm_{\C}}
\newcommand{\Smb}{\overline{\Sm}}
\newcommand{\Smcnis}{\Sm_{\C,\mathrm{Nis}}}

\newcommand{\Manc}{\mathbf{Man}_{\C}}

\newcommand{\sPre}{\mathbf{sPre}}

\newcommand{\Tb}{\mathbf{T}}
\newcommand{\sPreTb}{\sPre(\Tb)}
\newcommand{\hosPreTb}{\mathrm{ho}\sPreTb}

\newcommand{\hosPre}{\mathrm{ho}\sPre}

\newcommand{\hosPresm}{\mathrm{ho}\sPresm}

\newcommand{\hosPresmc}{\mathrm{ho}\sPresmc}
\newcommand{\sPresm}{\sPre(\Sm)}
\newcommand{\sPresmb}{\sPre(\Smb)}
\newcommand{\sPresmc}{\sPre(\Sm_{\C})}
\newcommand{\sS}{\mathbf{sS}}

\newcommand{\Ch}{{\mathcal C}}
\newcommand{\Dh}{{\mathcal D}}
\newcommand{\Eh}{{\mathcal E}}

\newcommand{\Fh}{{\mathcal F}}

\newcommand{\Gh}{{\mathcal G}}

\newcommand{\Homm}{{\mathcal Hom}}

\newcommand{\Ph}{\mathcal{P}}
\newcommand{\Qh}{\mathcal{Q}}

\newcommand{\Uh}{{\mathcal U}}
\newcommand{\Vh}{{\mathcal V}}
\newcommand{\Wh}{\mathcal{W}}
\newcommand{\Xh}{\mathcal{X}}

\newcommand{\Yh}{\mathcal{Y}}
\newcommand{\Zh}{\mathcal{Z}}
\newcommand{\oA}{\bar{A}}

\newcommand{\oX}{\overline{X}}
\newcommand{\oY}{\overline{Y}}

\newcommand{\oOmega}{\overline{\Omega}}

\newcommand{\into}{\hookrightarrow}

\newtheorem{theorem}{Theorem}[section]
\newtheorem{lemma}[theorem]{Lemma}
\newtheorem{prop}[theorem]{Proposition}

\theoremstyle{definition}
\newtheorem{defn}[theorem]{Definition}

\newtheorem{remark}[theorem]{Remark}

\begin{document}

\title{An Abel-Jacobi invariant for cobordant cycles}

\author{Gereon Quick}
\thanks{The author was supported in part by the German Research Foundation under QU 317/1-2 and RO 4754/1-1.} 

\address{Department of Mathematical Sciences, NTNU, NO-7491 Trondheim, Norway}
\email{gereon.quick@math.ntnu.no}
\date{}

\begin{abstract}
We discuss an Abel-Jacobi invariant for algebraic cobordism cycles whose image in topological cobordism vanishes. The existence of this invariant follows by abstract arguments from the construction of Hodge filtered cohomology theories in joint work of Michael J.\,Hopkins and the author. In this paper, we give a concrete description of the Abel-Jacobi map and Hodge filtered cohomology groups for projective smooth complex varieties.  
\end{abstract}

\maketitle

\section{Introduction}

The Abel-Jacobi map $\Phi$ is a fundamental invariant which appears in various forms in algebraic geometry. Let us mention three important examples. For an elliptic curve $E$ over the complex numbers, $\Phi$ is the map that identifies the group of complex valued points $E(\C)$ with a complex torus $\C/\Lambda$, where $\Lambda$ is a lattice defined by the periods of $E$. For general smooth projective complex varieties, $\Phi$ is a fundamental tool in Lefschetz' proof of the Hodge conjecture for $(1,1)$-classes. 
In his seminal work \cite{griffiths}, Griffiths showed that the Abel-Jacobi map can be used to detect cycles which may have codimension bigger than one and are homologous to zero. 
According to Deligne, one way to define the Abel-Jacobi map is the following.  
Let $H^{2p}_{\Dh}(X; \Z(p))$ denote the $2p$th Deligne cohomology of $X$ with coefficients in $\Z(p)$, let $\Hdg^{2p}(X)$ be the group of integral Hodge classes in $H^{2p}(X; \Z)$, and let $CH^p(X)$ be the $p$th Chow group of cycles of codimension $p$ modulo rational equivalence. Let $CH_{\hom}^p(X)$ be the subgroup of $CH^p(X)$ of cycles which are homologous to zero. 
Then there is a commutative diagram 
\bq\label{chdiagram}
\xymatrix{
 & CH_{\hom}^p(X) \ar[d]_{\Phi} \ar[r] & CH^p(X) \ar[d]_{\cl_{\Dh}} \ar[dr]^{\cl} & & \\
0 \ar[r] & J^{2p-1}(X) \ar[r] & H_{\Dh}^{2p}(X; \Z(p)) \ar[r] & \Hdg^{2p}(X) \ar[r] & 0.}
\eq
The bottom row of this diagram is an exact sequence, and the homomorphism $\Phi$ is induced by the Deligne-cycle map $\cl_{\Dh}$ and the fact that $CH_{\hom}^p(X)$ is the kernel of the cycle map $\cl$. The group $J^{2p-1}(X)$ is the $p$th intermediate Jacobian of Griffiths defined by  
\[
J^{2p-1}(X) = H^{2p-1}(X; \C)/\left( F^pH^{2p-1}(X; \C) + H^{2p-1}(X; \Z)\right).
\]

The purpose of this paper is to study an analog of the Abel-Jacobi map when we replace the role of Chow groups and cohomology with algebraic and complex cobordism, respectively. In \cite{lm}, Levine and Morel constructed algebraic cobordism as the universal oriented cohomology theory on the category $\Sm_k$ of smooth quasi-projective schemes over a field $k$ of characteristic zero. For $X\in \Sm_k$, the algebraic cobordism ring of $X$ is denoted by $\Omega^*(X)$. For a given $p\ge 0$, $\Omega^p(X)$ is generated by prime cycles of the form $f \colon Y \to X$ where $Y$ is a smooth scheme over $k$ and $f$ is a projective $k$-morphism. In the case $k=\C$, taking complex points induces a natural homomorphism of rings 
\[
\varphi_{MU} \colon \Omega^*(X) \to MU^{2*}(X) := MU^{2*}(X(\C))
\]               
to the complex cobordism, represented by the Thom spectrum $MU$, of the space of complex points $X(\C)$. 

More recently, Michael J.\,Hopkins and the author \cite{hfcbordism} constructed natural generalizations of Deligne-Beilinson cohomology on $\Smc$ for any topological spectrum $E$. We remark that a version of Hodge filtered complex $K$-theory had already been defined and studied by Karoubi in \cite{mk43} and \cite{mk45}. 

For $E=MU$, we obtain logarithmic Hodge filtered complex bordism groups. For $n, p \in \Z$ and $X \in \Smc$, they are denoted by $\MUlog^n(p)(X)$. Taking the sum over all $n$ and $p$, $\MUlog$ is equipped with a ring structure. Moreover, it was shown in \cite[\S 7.2]{hfcbordism} that $\MUlog^{2*}(*)(-)$ defines an oriented cohomology theory on $\Smc$. The universality of $\Omega^*(-)$ (proven in \cite{lm}) then implies that, for every $X \in \Smc$, there is a natural ring homomorphism 
\[
\varphi_{\MUlog} \colon \Omega^*(X) \to \MUlog^{2*}(*)(X).
\]

Furthermore, we can generalize diagram \eqref{chdiagram} in the following way.
For a given $p$, let $\Omega_{\topp}^p(X)$ be the kernel of the map $\varphi_{MU}$. Then for every smooth projective complex variety $X$, there is a natural commutative diagram   
\begin{equation}\label{omdiagram}
\xymatrix{
 & \Omtop^p(X) \ar[d]_{\Phi_{MU}} \ar[r] & \Omega^p(X) \ar[d]_{\varphi_{\MUlog}} \ar[dr]^{\varphi_{MU}} & & \\
0 \ar[r] & J_{MU}^{2p-1}(X) \ar[r] & \MUlog^{2p}(p)(X) \ar[r] & \HdgMU^{2p}(X) \ar[r] & 0.}
\end{equation}
The bottom row of this diagram is again exact. The group $\HdgMU^{2p}(X)$ is the subgroup of elements in $MU^{2p}(X)$ which are mapped to Hodge classes in cohomology. 
The group $J^{2p-1}_{MU}(X)$ is a complex torus which is determined by the Hodge structure of the cohomology and the complex cobordism of $X$. We consider $J^{2p-1}_{MU}(X)$ as a natural generalization of Griffiths' intermediate Jacobian. The map $\Phi_{MU}$ is induced by $\varphi_{\MUlog}$ and can be considered as an analog of the Abel-Jacobi map. 

By its definition via diagram \eqref{omdiagram}, $\Phi_{MU}$ can be used to detect algebraic cobordism cycles which are topologically cobordant. The main goal of this paper is to describe how one can associate to an element in $\Omtop^p(X)$ an element in $J_{MU}^{2p-1}(X)$. We will achieve this goal by providing a concrete description of the elements in Hodge filtered cohomology groups for {\it projective} smooth complex varieties.

We would like to add a few more words on the relationship between diagrams \eqref{chdiagram} and \eqref{omdiagram}. By \cite{lm} and \cite{hfcbordism}, there is a natural commutative diagram 
\bq\label{commsquareintro}
\xymatrix{
\Omega^p(X)\ar[d]_{\theta} \ar[r] & \MUlog^{2p}(p)(X) \ar[d]_{\vartheta_{\log}} \ar[r] & MU^{p}(X) \ar[d]_{\vartheta} \\
CH^p(X) \ar[r] & H^{2p}_{\Dh}(X; \Z(p)) \ar[r] & H^{2p}(X; \Z)}
\eq
The composite $\Omega^p(X) \to H^{2p}(X; \Z)$ is the canonical homomorphism induced by the transformation of oriented cohomology theories. Let $\Omega^p_{\hom}(X)$ be the subgroup of elements in $\Omega^p(X)$ which are mapped to zero under $\Omega^p(X) \to H^{2p}(X; \Z)$. It is clear from diagram \eqref{commsquareintro} that we have a natural inclusion 
\[
\Omtop^p(X) \subset \Omega^p_{\hom}(X).
\]  
We also have an induced map $\theta_{\hom} \colon \Omega^p_{\hom}(X) \to CH^p_{\hom}(X)$. An Abel-Jacobi map for $\Omega^p(X)$ which would correspond to the canonical map $\Omega^p(X) \to H^{2p}(X; \Z)$ would factor through $\theta_{\hom}$. But $\theta_{\hom}$ has a huge kernel to which such a map would be insensitive. The map $\Phi_{MU}$ however is a much finer invariant for algebraic cobordism. 

In fact, $\Phi_{MU}$ is able to detect at least some elements in the kernel of $\theta$. In \cite[\S 7.3]{hfcbordism}, we considered simple examples of algebraic cobordism classes in $\Omtop^*(X)$ which lie in the kernel of $\theta$. These classes can be obtained from the cycles which Griffiths constructed in \cite{griffiths} to show that the Griffiths group can be infinite. We hope  that this paper will help finding new types of examples.

Now let $E$ be another complex oriented cohomology theory and let $E^*_{\log}(*)$ be the associated logarithmic Hodge filtered cohomology theory. Let us assume that $E$ is equipped with maps $MU \to E \to H\Z$. This data induces an intermediate row in diagram \eqref{commsquareintro}. The natural map $\Omega^p(X) \to \Elog^{2p}(p)(X)$ factors through a quotient of $\Omega^p(X)$ which is determined by the formal group law of $E$. This yields the subgroup $\Omega^p_{E, \topp}(X)$ of those elements in $\Omega^p(X)$ which vanish under the natural map 
\[
\Omega^p(X) \to E^{2p}(X)=E^{2p}(X(\C)).
\] 
This subgroup satisfies 
\[
\Omtop^p(X) \subseteq \Omega^p_{E, \topp}(X) \subseteq \Omega^p_{\hom}(X).
\]
We do not discuss the difference of these groups in the present paper. However, we would like to remark that there is another interesting subgroup of $\Omega^p(X)$ which is given by algebraic cobordism cycles which are algebraically equivalent to zero in the sense of the work of Krishna and Park in \cite{kp}. This group sits between $\Omega^p_{\topp}(X)$ and $\Omega^p_{\hom}(X)$. 
Moreover, we expect the topological triviality relation we consider in this paper to be strictly coarser than the one in \cite{kp}. It would be very interesting to understand the difference between these relations in more detail. \\

We will now give a brief overview of the organization of the paper. 
In order to understand the map $\Phi_{MU}$ we provide a concrete representation of elements in logarithmic Hodge filtered cohomology theories which may be of independent interest. An element in $\MUlog^n(p)(X)$ can be represented by pairs of elements consisting of holomorphic forms and cobordism elements which are connected by a homotopy. This resembles the way one can view elements in Deligne cohomology for complex manifolds (see e.g.\,\cite{ev} or \cite{voisin1}) and elements in differential cohomology theories for smooth manifolds as in \cite{hs}. In order to obtain this representation we first discuss some facts about homotopy pullbacks for simplicial presheaves. 
Then we define logarithmic Hodge filtered spaces and study their global sections for smooth projective varieties. The above mentioned representation is then an immediate consequence of the construction. In the fourth section we use this representation to describe the Abel-Jacobi invariant for topologically trivial algebraic cobordism cycles. \\

{\bf Acknowledgements.} We are very grateful to Aravind Asok, Dustin Clausen, Mike Hopkins, Marc Levine and Kirsten Wickelgren for many helpful discussions. We would also like to thank the anonymous referee of the first version of this paper for very helpful suggestions and comments.

% % % % % % % % % % % % % % % % % % % % %

\section{Homotopy pullbacks of simplicial presheaves}

We briefly recall some basic facts about simplicial presheaves. Then we will discuss the construction of homotopy pullbacks of simplicial presheaves which will be needed in the next section.

\subsection{Simplicial presheaves}

Let $\Tb$ be an essentially small site with enough points. The following two examples of such sites  will occur in this paper: 

\begin{itemize}
\item The category $\Manc$ of complex manifolds and holomorphic maps which we consider as a site with the Grothendieck topology defined by open coverings.  

\item The category $\Smcnis=\Sm$ of smooth complex algebraic varieties (separated schemes of finite type over $\C$) with the Nisnevich topology. 
We recall that a distinguished square in $\Smcnis$ is a cartesian square of the form 
\begin{equation}\label{dsquare}
\xymatrix{U\times_XV\ar[r] \ar[d] & V \ar[d]^p \\
U \ar[r]^j & X}
\end{equation}
such that $p$ is an \'etale morphism, $j$ is an open embedding and the induced morphism $p^{-1}(X-U)\to X-U$ is an isomorphism, where the closed subsets are equipped with the reduced induced structure. The Nisnevich topology is the Grothendieck topology generated by coverings of the form \eqref{dsquare} (see \cite[\S 3.1]{mv}).
\end{itemize}

We denote by $\sPre=\sPreTb$ the category of simplicial presheaves on $\Tb$, i.e., contravariant functors from $\Tb$ to the category $\sS$ of simplicial sets. Objects in $\sPre$ will also be called spaces. There are several important model structures on the category $\sPre$ (see \cite{jardine}, \cite{blander}, \cite{dugger}).

We start with the projective model structure on $\sPre$. A map $\Fh \to \Gh$ in $\sPre$ is called  
\begin{itemize}
\item an objectwise weak equivalence if $\Fh(X) \to \Gh(X)$ is a weak equivalence in $\sS$ (equipped with the standard model structure) for every $X\in \Tb$;
\item a projective fibration if $\Fh(X) \to \Gh(X)$ is a Kan fibration in $\sS$ for every $X\in \Tb$;
\item a projective cofibration if it has the left lifting property with respect to all acyclic fibrations.
\end{itemize}

In order to obtain a local model structure, i.e., one which respects the topology on the site $\Tb$, we can localize the projective model structure at the hypercovers in $\Tb$ (see \cite{jardine}, \cite{blander}, \cite{dugger}). 
We briefly recall the most important notions. A map $f: \Fh \to \Gh$ of presheaves on $\Tb$ is called a {\it generalized cover} if for any map $X\to \Fh$ from a representable presheaf $X$ to $\Fh$ there is a covering sieve $R \into X$ such that for every element $U \to X$ in $R$ the composite $U\to X \to \Fh$ lifts through $f$. 

Dugger and Isaksen \cite[\S 7]{di} give the following characterization of local acyclic fibrations and hypercovers. A map $f: \Fh \to \Gh$ of simplicial presheaves on $\Tb$ is a {\it local acyclic fibration} if for every $X\in \Tb$ and every commutative diagram 
\[
\xymatrix{
\partial\Delta^n \otimes X \ar[r] \ar[d] & \Fh \ar[d]\\
\Delta^n\otimes X \ar[r] & \Gh}
\]
there exists a covering sieve $R\into X$ such that for every $U\to X$ in $R$, the diagram one obtains from restricting from $X$ to $U$ 
\[
\xymatrix{
\partial\Delta^n \otimes U \ar[r] \ar[d] & \Fh \ar[d]\\
\Delta^n\otimes U \ar@{.>}[ur] \ar[r] & \Gh}
\]
has a lifting $\Delta^n\otimes U \to \Fh$. Note that this implies in particular that the map $\Fh_0 \to \Gh_0$ of presheaves is a generalized cover.

\begin{defn}
Let $X$ be an object of $\Tb$ and let $\Uh$ be a simplicial presheaf on $\Tb$ with an augmentation map $\Uh \to X$ in $\sPre$. This map is called a {\it hypercover} of $X$ if it is a local acyclic fibration and each $\Uh_n$ is a coproduct of representables. 
\end{defn}

If $\Uh \to X$ is a hypercover, then the map $\Uh_0 \to X$ is a cover in the topology on $\Tb$. Moreover, the map $\Uh_1 \to \Uh_0 \times_X \Uh_0$ is a generalized cover. In general, for each $n$, the face maps combine such that $\Uh_n$ is a generalized cover of a finite fiber product of different $\Uh_k$ with $k < n$. 

Since the projective model structure on $\sPre$ is cellular, proper and simplicial, it admits a left Bousfield localization with respect to all maps 
\[
\{\hocolim \Uh_* \to X\}
\] 
where $X$ runs through all objects in $\Tb$ and $\Uh$ runs through the hypercovers of $X$. The resulting model structure is the {\it local projective model structure} on $\sPre$ (see \cite{blander} and \cite{dugger}). The weak equivalences, fibrations and cofibrations in the local projective model structure are called {\it local weak equivalences}, {\it local projective fibrations} and {\it local projective cofibrations}, respectively. We denote the corresponding homotopy category by $\hosPre$. Note that the local weak equivalences are precisely those maps $\Fh \to \Gh$ in $\sPre$ such that the induced map of stalks $\Fh_x \to \Gh_x$ is a weak equivalence in $\sS$ for every point $x$ in $\Tb$.

Dugger, Hollander and Isaksen showed that the fibrations in the local projective model structure on $\sPre$ have a nice characterization (see \cite[\S\S 3+7]{dhi}). Let $\Uh \to X$ be a hypercover in $\sPre$ and let $\Fh$ be a projective fibrant simplicial presheaf. Since each $\Uh_n$ is a coproduct of representables, we can form a product of simplicial sets $\prod_a \Fh((U_n)^a)$ where $a$ ranges over the representable summands of $\Uh_n$. The simplicial structure of $\Uh$ defines a cosimplicial diagram in $\sS$  
\[
\prod_a \Fh(U_0^a) \rightrightarrows \prod_a \Fh(U_1^a) \overset{\longrightarrow}{\underset{\longrightarrow}{\longrightarrow}} \cdots 
\]
The homotopy limit of this diagram is denoted by $\holim_{\Delta}\Fh(\Uh)$.

Following \cite[Definition 4.3]{dhi} we say that a simplicial presheaf $\Fh$ satisfies \emph{descent for a hypercover $\Uh \to X$} if there is a projective fibrant replacement $\Fh \to \Fh'$ such that the natural map 
\begin{equation}\label{descentmap}
\Fh'(X) \to \holim_{\Delta}\Fh'(\Uh)
\end{equation}
is a weak equivalence. It is easy to see that if $\Fh$ satisfies descent for a hypercover $\Uh \to X$, then the map \eqref{descentmap} is a weak equivalence for {\it every} projective fibrant replacement $\Fh \to \Fh'$. By \cite[Corollary 7.1]{dhi}, the local projective fibrant objects in $\sPre$ are exactly those simplicial presheaves which are projective fibrant and satisfy descent with respect to all hypercovers $\Uh \to X$. For our final applications we will need the following facts whose proofs can be found in \cite{compact}.

\begin{lemma}\label{descent}
Let $\Fh$ be a simplicial presheaf that satisfies descent with respect to all hypercovers. Then every fibrant replacement $\Fh \to \Fh_f$ in the local projective model structure is an objectwise weak equivalence, i.e., for every object $X\in \Tb$ the map 
\[
\Fh(X) \to \Fh_f(X)
\]
is a weak equivalence of simplicial sets.
\end{lemma}

\begin{prop}\label{cordescent}
Let $\Fh$ be a simplicial presheaf that satisfies descent with respect to all hypercovers and let $X$ be an object of $\Tb$. Then, for every projective fibrant replacement $g:\Fh \to \Fh'$, the natural map 
\[
\Hom_{\hosPre}(X, \Fh) \to \pi_0(\Fh'(X))\]
is a bijection.  
\end{prop}

% % % % % % % % % % % % % % % % % % % % % % %

\subsection{Homotopy pullbacks of simplicial presheaves}

We briefly recall the construction of homotopy pullbacks in $\sPre$ (see \cite[\S 13.3]{hirschhorn} for more details) and will then show that its local and global versions are homotopy equivalent. 

Let $\sPre$ be equipped with any of the above model structures. We fix a functorial factorization $\Eh$ of every map $f \colon \Xh \to \Yh$ into 
\[
\Xh \xrightarrow{i_f} \Eh(f) \xrightarrow{p_f} \Yh
\]
where $i_f$ is an acyclic cofibration and $p_f$ is a fibration. The {\em homotopy pullback} of the diagram 
$\Xh \xrightarrow{f} \Zh \xleftarrow{g} \Yh$ 
is the pullback of  
$\Eh(f) \xrightarrow{p_f} \Zh \xleftarrow{p_g} \Eh(g)$. The homotopy pullback satisfies the following invariance. If we have a diagram 
\[
\xymatrix{
\Xh \ar[d] \ar[r]^f & \Zh \ar[d] & \ar[l]_g \Yh \ar[d] \\
\Xh' \ar[r]^-{f'} & \Zh' & \ar[l]_{g'} \Yh' }
\]
in which the three vertical maps are weak equivalences, then the induced map of homotopy pullbacks 
\[
\Eh(f) \times_{\Zh} \Eh(g) \to \Eh(f') \times_{\Zh'} \Eh'(g')
\]
is a weak equivalence as well.

We will need the following fact about pulling back local weak equivalences along maps which are merely projective fibrations.  

\begin{lemma}\label{localproper}
Let $f \colon \Xh \to \Zh$ be a projective fibration and $g \colon \Yh \to \Zh$ be a local weak equivalence in $\sPre$. Then the induced map $f' \colon \Xh \times_{\Zh} \Yh \to \Xh$ is a local weak equivalence as well.
\end{lemma}
\begin{proof}
For a point $x$ of $\sPre$ and a map $h \colon \Vh \to \Wh$ of simplicial presheaves, let $h_x \colon \Vh_x \to \Wh_x$ denote the induced map of stalks at $x$. With the notation of the lemma, we need to show that $f'_x$ is a weak equivalence in $\sS$ for every point $x$ in $\Tb$. Since $x$ preserves finite limits, we have $(\Xh \times_{\Zh} \Yh)_x = \Xh_x \times_{\Zh_x} \Yh_x$ and $f'_x$ equals the induced map in the corresponding pullback diagram in $\sS$. Since the standard model structure on $\sS$ is right proper, it thus suffices to show that $g_x$ is a Kan fibration. But every map in $\sPre$ which is an objectwise Kan fibration is also a stalkwise Kan fibration. 
We will provide a proof of this fact for completeness. Given a point $x$ of $\Tb$, we need to check that the map of sets induced by $g$ 
\bq\label{stalk1}
\Hom_{\sS}(\Delta[n], \Yh_x) \to \Hom_{\sS}(\Lambda_k[n], \Yh_x) \times_{\Hom_{\sS}(\Lambda_k[n], \Zh_x)} \Hom_{\sS}(\Delta[n], \Zh_x) 
\eq
is surjective for all $n\ge 1$ and $0 \le k \le n$. Now for a simplicial presheaf $\Wh$ and a finite simplicial set $K$, we can consider the functor $X \mapsto \Hom_{\sS}(K, \Wh(X))$ as a presheaf of sets on $\Tb$. We denote this presheaf by $\Homm(K, \Wh)$. The stalk of this presheaf at $x$ is exactly the set $\Hom_{\sS}(K, \Wh_x)$. Hence the map \eqref{stalk1} is surjective if and only if the map of presheaves of sets 
\bq\label{stalk2}
\Homm(\Delta[n], \Yh) \to \Homm(\Lambda_k[n], \Yh) \times_{\Homm(\Lambda_k[n], \Zh)} \Homm(\Delta[n], \Zh)
\eq
induces a surjective map of stalks at $x$. But, since $g$ is a projective fibration, $g(X)$ is a Kan fibration for every object $X \in \Tb$. Hence, by the definition of $\Homm(-,-)$, the induced map 
\bq\label{stalk3}
\Homm(\Delta[n], \Yh)(X) \to \Homm(\Lambda_k[n], \Yh)(X) \times_{\Homm(\Lambda_k[n], \Zh)(X)} \Homm(\Delta[n], \Zh)(X)
\eq
is surjective. Since forming stalks preserves objectwise epimorphisms, this implies that $g_x$ is a Kan fibration.  
\end{proof}

The following result is probably a well-known fact. We include its proof for completeness and lack of a reference.

\begin{lemma}\label{htpypblemma}
The homotopy pullback of a diagram in $\sPre$ in the projective model structure is stalkwise weakly equivalent to the homotopy pullback of the diagram in the {\em local} projective model structure.  
\end{lemma}
\begin{proof}
Let $\Xh \xrightarrow{f} \Zh \xleftarrow{g} \Yh$ be a diagram in $\sPre$. Consider the diagram 
\bq\label{htpycompare}
\xymatrix{
\Xh \ar[d]^{i_f^{\proj}} \ar[r]^f & \Zh \ar@{=}[d] & \ar[l]_g \Yh \ar[d]^{i_g^{\proj}} \\
\Eh^{\proj}(f) \ar[d] \ar[r]^-{p_f^{\proj}} & \Zh \ar@{=}[d] & \ar[l]_{p_g^{\proj}} \Eh^{\proj}(g) \ar[d] \\
\Eh^{\loc}(p_f^{\proj}) \ar[r] & \Zh & \ar[l] \Eh^{\loc}(p_g^{\proj}) }
\eq
where the superscripts \emph{proj} and \emph{loc} indicate whether we take replacements in the projective and local projective model structure, respectively. We denote the projective homotopy pullback of the initial diagram by $\Ph^{\proj} = \Eh^{\proj}(f)\times_{\Zh} \Eh^{\proj}(g)$. The vertical maps $i_f^{\proj}$ and $i_g^{\proj}$ are projective acyclic cofibrations. Since left Bousfield localization does not change cofibrations and since objectwise weak equivalences are in particular local weak equivalences, $i_f^{\proj}$ and $i_g^{\proj}$ are local acyclic cofibrations as well. Hence the composition of the vertical maps in \eqref{htpycompare} are local acyclic cofibrations. Thus, $\Eh^{\loc}(p_f^{\proj}) \times_{\Zh} \Eh^{\loc}(p_g^{\proj})$ computes the local homotopy pullback $\Ph^{\loc}$ of $\Xh \xrightarrow{f} \Zh \xleftarrow{g} \Yh$. 
Hence we need to show that the induced map 
\bq\label{Pprojtoloc}
\Eh^{\proj}(f) \times_{\Zh} \Eh^{\proj}(g) \to \Eh^{\loc}(p_f^{\proj}) \times_{\Zh} \Eh^{\loc}(p_g^{\proj}) 
\eq
is a local weak equivalence. This map equals the composition 
\bq\label{twomaps}
\Eh^{\proj}(f) \times_{\Zh} \Eh^{\proj}(g) \to \Eh^{\loc}(p_f^{\proj}) \times_{\Zh} \Eh^{\proj}(g) \to \Eh^{\loc}(p_f^{\proj}) \times_{\Zh} \Eh^{loc}(p_g^{\proj}).
\eq
Hence in order to show that \eqref{Pprojtoloc} is a local weak equivalence, it suffices to show that the two maps in \eqref{twomaps} are both local weak equivalences. For this, it suffices to show that the pullback of a local weak equivalence along a projective fibration is again a local weak equivalence which has been checked in Lemma \ref{localproper}. 
\end{proof}

\begin{remark}
The result of Lemma \ref{htpypblemma} does not depend on the fact that we use the projective model structure. The same proof (after replacing the superscript \emph{proj} with \emph{inj}) would also work if we used the injective model structure on $\sPre$. More precisely, the homotopy pullback of a diagram in $\sPre$ in the injective model structure is stalkwise weakly equivalent to the homotopy pullback of the diagram in the local injective model structure. 
\end{remark}

The lemma shows that we can calculate the set of homotopy classes of maps into a homotopy pullback via global sections in the following way.  

\begin{prop}\label{htpypbprop}
Let $\Xh \xrightarrow{f} \Zh \xleftarrow{g} \Yh$ be a diagram in $\sPre$, and let $\Ph$ denote the homotopy pullback of this diagram in the local projective model structure. We assume that all three simplicial presheaves $\Xh$, $\Yh$ and $\Zh$ satisfy descent for all hypercovers. For an object $X \in \Tb$, let $\Qh(X)$ denote the homotopy pullback in $\sS$ of the diagram of simplicial sets $\Xh(X) \xrightarrow{f(X)} \Zh(X) \xleftarrow{g(X)} \Yh(X)$. 
Then there is a natural bijection for every $X \in \Tb$ 
\[
\Hom_{\hosPre}(X, \Ph) \cong \pi_0(\Qh(X)). 
\]
\end{prop}
\begin{proof} 
Let $\Xh \mapsto \Xh'$ be a functorial projective fibrant replacement in $\sPre$. The invariance property of homotopy pullbacks implies that the homotopy pullback of $\Xh \xrightarrow{f} \Zh \xleftarrow{g} \Yh$ is stalkwise equivalent to the homotopy pullback of the induced diagram $\Xh' \xrightarrow{f'} \Zh' \xleftarrow{g'} \Yh'$. It also implies that, for every $X \in \Tb$, $\Qh(X)$ is equivalent to the homotopy pullback $\Qh'(X)$ of $\Xh'(X) \xrightarrow{f'(X)} \Zh'(X) \xleftarrow{g'(X)} \Yh'(X)$ in $\sS$. Hence we can assume from now on that $\Xh$, $\Yh$ and $\Zh$ are also projective fibrant. 
Now consider the diagram 
\bq\label{globalreplacement}
\xymatrix{
\Xh \ar[d] \ar[r]^f & \Zh \ar@{=}[d] & \ar[l]_g \Yh \ar[d] \\
\Eh^{\proj}(f) \ar[r]^-{p_{f}^{\proj}} & \Zh & \ar[l]_{p_{g}^{\proj}} \Eh^{\proj}(g)}
\eq
where $\Eh^{\proj}$ is a functorial replacement in the projective model structure as before. Let $\Ph^{\proj} = \Eh^{\proj}(f) \times_{\Zh} \Eh^{\proj}(g)$ denote the homotopy pullback of $\Xh \xrightarrow{f} \Zh \xleftarrow{g} \Yh$ calculated in the projective model structure. By definition of pullbacks in $\sPre$, we have $\Ph^{\proj}(X) = \Eh^{\proj}(f)(X) \times_{\Zh(X)} \Eh^{\proj}(g)(X)$ for every $X \in \Tb$. 
The invariance property of homotopy pullbacks implies that $\Ph^{\proj}(X)$ is equivalent to the homotopy pullback $\Qh(X)$ of the diagram $\Xh(X) \xrightarrow{f(X)} \Zh(X) \xleftarrow{g(X)} \Yh(X)$ in $\sS$. (In fact, we could compute $\Qh(X)$ as $\Ph^{\proj}(X)$.)
Moreover, since $\Xh$, $\Yh$ and $\Zh$ satisfy descent for all hypercovers and since homotopy pullbacks commute with homotopy limits in $\sS$, we see that $\Ph^{\proj}$ satisfies descent for all hypercovers as well. By Lemma \ref{descent}, this implies that $\Ph^{\proj}$ is local projective fibrant. 
Finally, by Lemma \ref{htpypblemma}, $\Ph^{\proj}$ is equivalent to the homotopy pullback in the local projective model structure. Hence, by Proposition \ref{cordescent}, for every $X \in \Tb$, there are natural bijections  
\[
\Hom_{\hosPre}(X, \Ph) \cong \Hom_{\hosPre}(X, \Ph^{\proj}) \cong \pi_0(\Ph^{\proj}(X)) \cong \pi_0(\Qh(X)). 
\]
\end{proof}

\section{Logarithmic Hodge filtered function spaces}

We construct spaces which represent logarithmic Hodge filtered cohomology groups. In particular, we will show that we can represent elements in logarithmic Hodge filtered complex bordism groups as triples consisting of a class in complex bordism, a holomorphic form with suitable coefficients and a homotopy that connects both in an appropriate sense.

\subsection{Hodge filtration on forms and Eilenberg-MacLane spaces}

Let $\Ch^*$ be a cochain complex of presheaves of abelian groups on $\Tb$. 
For any given $n$, we denote by $\Ch^*[n]$ the cochain complex given in degree $q$ by $\Ch^q[n]:=\Ch^{q+n}$. The differential on $\Ch^*[n]$ is the one of $\Ch^*$ multiplied by $(-1)^n$. 
The hypercohomology $H^*(U, \Ch^*)$ of an object $U$ of $\Tb$ with coefficients in $\Ch^*$ is the graded group of morphisms $\Hom(\Z_U,a\Ch^*)$ in the derived category of cochain complexes of sheaves on $\Tb$, where $a\Ch^*$ denotes the complex of associated sheaves of $\Ch^*$. We will denote by $K(\Ch^*, n)$ the Eilenberg-MacLane simplicial presheaf corresponding to $\Ch^*[-n]$. 
The following result is a version of Verdier's hypercovering theorem due to Ken Brown.

\begin{prop}\label{verdier} {\rm (\cite[Theorem 2]{brown}, see also \cite{mv}, \cite{jardineverdier})}
Let $\Ch^*$ be a cochain complex of presheaves of abelian groups on $\Tb$. Then for any integer $n$ and any object $U$ of $\Tb$, one has a canonical isomorphism 
\[
H^n(U; \Ch^*)\cong \Hom_{\hosPreTb}(U, K(\Ch^*, n)).
\]
\end{prop}

% % % % % % % % % % % % % % % % % % % % % % % % % % % % %

Now let $\Tb$ be the site $\Sm$ of smooth complex varieties with the Nisnevich topology. 
We would like to find a simplicial presheaf which represents Hodge filtered complex cohomology. To have a good filtration on holomorphic forms requires a compact variety. By the work of Hironaka, we know that every smooth complex variety does have a nice compactification. Following Deligne \cite{hodge2} and Beilinson \cite{beilinson}, we will use this fact to construct simplicial presheaves on $\Sm$ whose global sections are isomorphic to the Hodge filtered cohomology groups of $X$.

Let $\Smb$ be the category whose objects are {\em smooth
compactifications}, i.e., pairs $(X,\oX)=(X\subset \oX)$ consisting of a
smooth variety $X$ embedded as an open subset of a projective variety
$\oX$ and having the property that $\oX-X$ is a normal crossing divisor which is the union of smooth divisors.  A map from $(X, \oX)$ to $(Y,\oY)$ is a commutative diagram 
\[
\xymatrix{
X  \ar[r]\ar[d]  &  \oX \ar[d] \\
Y  \ar[r]        &\oY.}
\]

By Hironaka's theorem \cite{hironaka}, every smooth variety over $\C$ admits a smooth compactification. Moreover, for a given smooth variety $X$, the category $C(X)$ of all smooth compactifications of $X$ is filtered (see \cite{hodge2}).

The forgetful functor 
\begin{align*}
u:\Smb & \to \Sm \\
(X, \oX) &\mapsto X
\end{align*}
induces a pair of adjoint functors on the categories of simplicial presheaves  
\[
u^{\ast}:\sPresmb \leftrightarrow \sPresm:u_{\ast}. 
\]
The left adjoint $u^*$ is given by sending a simplicial presheaf $\Fh$ on $\Smb$ to the simplicial presheaf 
\[
X \mapsto u^*\Fh(X) = \colim_{C(X)} \Fh(\oX).
\]

For a smooth complex variety $X$, let $\Omega^p_{X}$ denote the sheaf of holomorphic $p$-forms on $X$. 
Let $\oX$ be a smooth compactification of $X$ and let $D:=\oX - X$ denote the complement of $X$. Let $\Omega^1_{\oX}\langle D \rangle$ be the locally free sub-module of $j_*\Omega^1_{X}$ generated by $\Omega^1_{X}$ and by $\frac{dz_i}{z_i}$ where $z_i$ is a local equation for an irreducible local component of $D$. The sheaf $\Omega^p_{\oX}\langle D \rangle$ of meromorphic $p$-forms on $\oX$ with at most logarithmic poles along $D$ is defined to be the locally free subsheaf $\bigwedge^p\Omega^1_{\oX}\langle D \rangle$ of $j_*\Omega_{\oX}^p$. 
The Hodge filtration on the complex cohomology of $X$ can be defined as the image 
\begin{equation}\label{Hodgefiltration} 
F^pH^n(X; \C):= \Imm (H^n(\oX; \Omega^{* \geq p}_{\oX}\langle D \rangle) \to H^n(X;\C)).
\end{equation}
This definition is independent of the compactification $\oX$ (see \cite{hodge2}).

We denote by $\oOmega^*$ the presheaf of differential graded $\C$-algebras on $\Smb$ that sends a pair $X \subset \oX$ with $D:=\oX - X$ to $\Omega_{\oX}^* \langle D \rangle (\oX)$. For any given integer $p\geq 0$, we denote by $\oOmega^{*\geq p}$ the presheaf on $\Smb$ that sends a pair $X\subset \oX$ to $\Omega_{\oX}^{*\geq p} \langle D \rangle (\oX)$.

Let 
\[
\Omega^{\ast\ge p}_{\oX}\langle D \rangle \to A^{*\geq p}_{\oX}\langle D \rangle
\]
be any resolution by cohomologically trivial sheaves which is
functorial in $X\subset \oX$ and which induces a commutative diagram 
\[
\xymatrix{
\Omega^{\ast\ge p}_{\oX}\langle D \rangle (\oX) \ar[r] \ar[d] & \Omega^{*\ge p}_X(X) \ar[d] \\
A^{*\geq p}_{\oX}\langle D \rangle (\oX) \ar[r] & A^{*\ge p}_X(X)}
\]
where $A_X^{*\ge p}$ denotes a functorial resolution by cohomologically trivial sheaves of $\Omega_X^{*\ge p}$.  
For example, $A^{*\geq p}_{\oX}\langle D \rangle$ and $A^{*\ge p}_X$ could be the Godemont resolutions (\cite[\S 3.2.3]{hodge2}) or the logarithmic Dolbeault resolution (\cite[\S 8]{navarro}). 
Even though $A_X^*$ and $A^{*\geq p}_{\oX}\langle D \rangle$ are double complexes, we will only consider their total complexes.

We denote the presheaf of complexes on $\Smb$ that sends a pair $(X,\oX)$ to $A^{*\geq p}_{\oX}\langle D \rangle (\oX)$ by $F^p\oA^{*}$, and let 
\[
\oOmega^{\ast\ge p}  \to F^p\oA^{*} 
\]
be the associated map of complexes of presheaves on $\Smb$.

Now let $\Vh_{*}$ be an evenly graded $\C$-algebra such that each $\Vh_{2j}$ is a finite dimensional complex vector space. We will write 
\[
F^pH^n(X; \Vh_{*}) := \bigoplus_j F^{p+j}H^{n+2j}(X; \Vh_{2j})
\]
for the graded Hodge filtered cohomology groups. 

The functor $X \mapsto F^pH^n(X; \Vh_{*})$ is representable in $\hosPresm$ in the following way.
For $(X, \oX) \in \Smb$ and $j \in \Z$, let $F^{p+j}\oA^*(\oX; \Vh_{2j})[-2j]$ denote the corresponding complex with coefficients in $\Vh_{2j}$ shifted by degree $2j$. 
We write 
\bq\label{defofFpA*graded}
F^p\oA^*(\oX; \Vh_{*}) = \bigoplus_j F^{p+j}\oA^*(\oX; \Vh_{2j})[-2j].
\eq
Let $F^p\oA^*(\Vh_{*})$ denote the corresponding presheaf on $\Smb$. 
Let $K(F^p\oA^*(\Vh_{*}), n)$ be the associated Eilenberg-MacLane simplicial presheaf. Note that \eqref{defofFpA*graded} induces an isomorphism 
\bq\label{KofFpA*graded}
K(F^p\oA^*(\Vh_{*}), n) \cong \bigvee_j K(F^{p+j}\oA^*(\Vh_{2j}), n+2j).
\eq

For every smooth complex variety $X$, \cite[Theorem 3.5]{compact} shows that there is a natural isomorphism 
\[
\Hom_{\hosPresm}(X, u^*K(F^p\oA^{*}(\Vh_{*}), n)) \cong F^pH^n(X; \Vh_{*}).
\]

A crucial fact for the proof of \cite[Theorem 3.5]{compact} is that the simplicial presheaf $u^*K(F^p\oA^{*}(\Vh_{*}), n)$ satisfies Nisnevich descent. This implies that any projective fibrant replacement of $u^*K(F^p\oA^{*}(\Vh_{*}), n)$ is already local projective fibrant. As a consequence we obtain that,
for every smooth complex variety $X$, there is a natural isomorphism 
\begin{equation}\label{3.5pi}
\pi_0(u^*K(F^p\oA^{*}(\Vh_{*}), n)(X)) \cong F^pH^n(X; \Vh_{*}).
\end{equation}

Finally, we point out that, for every $n$ and $p$, the map of presheaves of complexes 
\[
F^p\oA^*(\Vh_{*})[-n] \to A^*(\Vh_{*})[-n]
\]
induces a morphism of simplicial presheaves
\bq\label{mapofEMspaces}
u^*K(F^p\oA^*(\Vh_{*}), n) \to K(A^*(\Vh_{*}), n).
\eq

%%%%%%%%%%%%%%%%%%%%%%%%%

\subsection{The singular functor for complex manifolds}

As a short digression, we need to consider simplicial presheaves on complex manifolds as well. 
In this subsection, we let $\Tb$ be the site $\Manc$.   
Let $\DD^n$ be the standard topological $n$-simplex 
\[
\DD^n =\{ (t_0,\ldots, t_n) \in \R^{n+1} | 0 \leq t_j \leq 1, \sum t_j =1 \}.
\]
For topological spaces $Y$ and $Z$, the singular function complex $\Sing(Z,Y)$ is the simplicial set whose $n$-simplices are continuous maps
\[
f \colon Z \times \DD^n \to Y.
\]
We denote the simplicial presheaf  
\[
M \mapsto \Sing(M,Y)=:\Sing Y(M)
\]
on $\Manc$ by $\Sing Y$. 
Since, for any $CW$-complex $Y$, $\Sing Y$ satisfies descent, the criterion of \cite{dhi} implies that $\Sing Y$ is a fibrant object in the local projective model structure on $\sPre$ (see also \cite[Lemma 2.3]{hfcbordism}).

Let $\Vh_{*}$ be an evenly graded complex vector space, and let $K(\Vh_{*}, n)$ be an associated  Eilenberg-MacLane space in the category of $CW$-complexes. Then the simplicial presheaf $\Sing K(\Vh_{*}, n)$ represents the functor of cocycles with coefficients in $\Vh_*$, i.e., for every $M \in \Manc$, there is a natural isomorphism of abelian groups
\[
Z^n(M; \Vh_{*}) \cong \Hom_{\sPre}(M, \Sing K(\Vh_{*}, n)). 
\]
Since $M$ is a representable presheaf, we have a natural bijection of sets 
\[
\Hom_{\sPre}(M, \Sing K(\Vh_{*}, n)) = \mathrm{Sing}_0 K(\Vh_{*}, n)(M).
\]
Moreover, $M$ is a cofibrant object in the local projective model structure on $\sPre$. Hence there is a natural bijection
\bq\label{pi0EMcocycles}
\Hom_{\hosPre}(M, \Sing K(\Vh_{*}, n)) = \pi_0(\Sing K(\Vh_{*}, n)(M)).
\eq

% % % % % % % % % % % % % % % % % % % % % %

\subsection{Logarithmic Hodge filtered function spaces}\label{HFS}

In this subsection we will work with both sites, $\Smc$ and $\Manc$. 
For $X\in \Smc$, we denote by $X_{\an} \in \Manc$ the associated complex manifold whose underlying set is $X(\C)$.  
This defines a functor
\[
\rho^{-1} \colon \Smc \to \Manc, ~ X \mapsto \rho^{-1}(X):= X_{\an}.
\]
Composition with $\rho^{-1}$ induces a functor 
\[
\rho_* \colon \sPre(\Manc) \to \sPre(\Smc).
\]
Note that $\rho_*$ is the right adjoint in a Quillen pair of functors between the corresponding local projective model structures.

We can now construct logarithmic Hodge filtered spaces whose global sections yield generalized Hodge filtered cohomology groups. 
The idea to define Hodge filtered spaces is similar to the way that differential function spaces were defined for presheaves on the category of smooth manifolds in \cite{hs}.

Let $n$, $p$ be integers and $\Vh_{*}$ an evenly-graded complex vector space. 
Let $Y$ be a CW-complex and let $\iota \in Z^n(Y; \Vh_{*})$ by a cocycle on $Y$. 
A cocycle corresponds to a map of CW-complexes 
\[
Y \to K(\Vh_{*}, n)
\]
and induces a map of simplicial presheaves on $\Manc$
\[
\Sing Y \to \Sing K(\Vh_{*}, n).
\]
Let $|\cdot|$ denote the geometric realization of simplicial sets. 
Using the canonical map $K(\Vh_{*}, n) \to |K(A^*(\Vh_{*}), n)|$ we can form the following diagram in $\sPre(\Smc)$
\bq\label{firstdef}
\xymatrix{
 & \rho_*\Sing Y \ar[d]^{\iota^*}\\
u^*K(F^p\oA^*(\Vh_{*}), n) \ar[r] & \rho_*\Sing |K(A^*(\Vh_{*}), n)|.}
\eq

\begin{defn}
We define the \emph{logarithmic Hodge filtered function space $(Y(p), \iota, n)$} to be the homotopy pullback of \eqref{firstdef} in $\sPre(\Smc)$. 
\end{defn}

Note that $(Y(p), \iota, n)$ depends on $\iota$ only up to homotopy, i.e., if $\iota'$ is another cocycle which represents the same cohomology class as $\iota$ then $(Y(p), \iota, n)$ and $(Y(p), \iota', n)$ are equivalent.

\begin{remark}\label{Irem}
Let us contemplate a little more on diagram \eqref{firstdef}. For a complex manifold $M$, let $Z^n(M\times \Delta^{\bullet}; \Vh_{*})$ be the simplicial abelian group whose group of $k$-simplices is given by $C^{\infty}$-$n$-cocycles on $M \times \Delta^k$ with coefficients in $\Vh_{*}$. We denote the corresponding simplicial presheaf 
\[
M\mapsto Z^n(M \times \Delta^{\bullet}; \Vh_{*})
\]
on $\Manc$ by $Z^n(- \times \Delta^{\bullet}; \Vh_{*})$. 
Our chosen cocycle $\iota$ determines a map of simplicial presheaves
\[
\Sing Y \to Z^n(- \times \Delta^{\bullet}; \Vh_{*}), ~ f \mapsto \iota^*f, 
\]
given by taking the pullback along $\iota$. 
Let $I$ denote the map given by integration of forms
\[
I \colon F^{p+j}A^{n+2j}(X; \Vh_{2j}) \to C^{n+2j}(X; \Vh_{2j}), ~ \eta \mapsto (\sigma \mapsto \int_{\Delta^{n+2j}}\sigma^*\eta).
\]
We can form a diagram of simplicial presheaves 
\bq\label{seconddef}
\xymatrix{
 & \rho_*\Sing Y \ar[d]^{\iota^*}\\
u^*K(F^p\oA^*(\Vh_{*}), n) \ar[r]_I & \rho_*Z^n(- \times \Delta^{\bullet}; \Vh_{*}).}
\eq

The map 
\[
\Sing K(\Vh_{*}, n)(M) \to Z^n(M \times \Delta^{\bullet}; \Vh_{*})
\]
given by pulling back a fundamental cocycle in $Z^n(K(\Vh_{*}, n); \Vh_{*})$ is a simplicial homotopy equivalence (see e.g.\,\cite[Proposition A.12]{hs}). Hence the homotopy pullback of \eqref{seconddef} is homotopy equivalent to the homotopy pullback of \eqref{firstdef}. 
\end{remark}

\begin{remark}\label{remarkDeligne1}
For $Y=K(\Z, n)$, we recover Deligne-Beilinson cohomology in the following way. Let $\iota \colon K(\Z, n) \to K(\C, n)$ be the map that is induced by the $(2\pi i)^p$-multiple of the inclusion $\Z \subset \C$. Then $K(\Z, n)(p) := (K(\Z, n)(p), \iota, n)$ represents Deligne-Beilinson cohomology in the homotopy category of $\sPresmc$ in the sense that there is a natural isomorphism 
\[
H^n_{\Dh}(X; \Z(p)) \cong \Hom_{\hosPresmc}(X, K(\Z, n)(p)).
\]
\end{remark}

\subsection{Hodge filtered spaces and spectra}

Of particular interest is the case when $Y$ is a space in a spectrum. We refer the reader to \cite{hfcbordism} for a more detailed discussion of the maps of spectra involved. We can reinterpret the construction of \cite{hfcbordism} on the level of spaces as follows.  

Let $E$ be a topological $\Omega$-spectrum and let $E_n$ be its $n$th space. We assume that $E$ is rationally even, i.e., $\pi_*E\otimes \Q$ is concentrated in even degrees. Let $\Vh_{*}$ be the evenly graded $\C$-vector space $\pi_{*}E\otimes \C$. 
Let 
\[
\tau \colon E \to E \wedge H\C=:E_{\C}
\]
be a map of spectra which induces for every $n$ the map 
\[
\pi_{2n}(E) \xrightarrow{(2\pi i)^n} \pi_{2n}(E_{\C})
\]
defined by multiplication by $(2\pi i)^n$ on homotopy groups. The choice of such a map is unique up to homotopy. 
For a given integer $p$, multiplication by $(2\pi i)^p$ on homotopy groups determines a map  
\[
E \xrightarrow{(2\pi i)^p\tau} E \wedge H\C.
\]
Let  
\[
E\wedge H\C \to H(\pi_{*}E\otimes \C)
\]
be a map that induces the isomorphism  
\[
\pi_{*}(E\wedge H\C) \cong \pi_{*}E \otimes \C =\Vh_*.
\]
The composite with $(2\pi i)^p\tau$ defines a map of spectra 
\[
\iota \colon E \to H(\Vh_*). 
\]
The inclusion $\Vh_* \into A^*(\Vh_{*})$ induces a map of spectra $H(\Vh_*) \to H(A^*(\Vh_{*}))$. Composition with $E \to H(\Vh_{*})$ defines a map 
\bq\label{iotaEA}
\iota \colon E \to H(A^*(\Vh_{*}))
\eq
which we also denote by $\iota$. We call $\iota$ a \emph{$p$-twisted fundamental cocycle of $E$}.

This map corresponds to a family of maps of spaces which, for each $n$, are of the form 
\[
\iota_n \colon E_n \to K(A^*(\Vh_{*}), n)
\]
and which are compatible with the structure maps of the spectrum $E$.

For given $p$ and $\iota$ and each $n$, we can form the diagram in $\sPre(\Smc)$
\bq\label{Efirstdef}
\xymatrix{
 & \rho_*\Sing E_n \ar[d]^{\iota_n^*}\\
u^*K(F^p\oA^*(\Vh_{*}), n) \ar[r] & \rho_*\Sing |K(A^*(\Vh_{*}), n)|.}
\eq
We will write $(E_n(p), \iota)$ for the homotopy pullback of \eqref{Efirstdef} in $\sPre(\Smc)$. Note that a different choice $\iota'$ of a $p$-twisted fundamental cocycle of $E$ yields a homotopy equivalent simplicial presheaf $(E_n(p), \iota')$. Therefore, we will often drop $\iota$ from the notation and write $E_n(p)$ for $(E_n(p), \iota)$.

\begin{defn}
According to our previous terminology, we call $E_n(p)$ the \emph{$n$th logarithmic Hodge filtered function space of $E$} (even though it is only unique up to homotopy equivalence). 
\end{defn}

\begin{remark}\label{Elogiso}
For $X \in \Smc$, let $\Elog^n(p)(X)$ denote the logarithmic Hodge filtered $E$-cohomology groups of $X$ as defined in \cite[Definition 6.4]{hfcbordism}. It follows from the definition of $E_n(p)$ as a homotopy pullback of \eqref{Efirstdef} that the groups $\Hom_{\hosPre}(X, E_n(p))$, for varying $n$, sit in long exact sequences analog to the one of \cite[Proposition 6.5]{hfcbordism}. This shows that we have a natural isomorphism  
\[
\Elog^n(p)(X) \cong \Hom_{\hosPre}(X, E_n(p)).
\]
Alternatively, we could have remarked that $E_n(p)$ is the $n$th space of the fibrant spectrum $\Elog(p)$ of \cite[\S 6]{hfcbordism}.
\end{remark}

% % % % % % % % % % % % % % % % % % % % % %

\subsection{The case of smooth projective varieties}\label{projectivecase}

If $X$ is a \emph{projective} smooth complex varieties, we obtain a more concrete description of the global sections of a Hodge filtered space. 
For, in this case, $X$ is an initial object in the filtered category $C(X)$ of all smooth compactifications of $X$. Hence the colimit that computes the value of $u^*K(F^p\oA^{*}(\Vh_{*}), n)$ at $X$ reduces to
\[
u^*K(F^p\oA^{*}(\Vh_{*}), n)(X) = K(F^pA^{*}(\Vh_{*}), n)(X).
\] 
Thus, for $X$ projective, we have 
\bq\label{3.5projectivecocycles2}
\Hom_{\sPresmc}(X, u^*K(F^p\oA^{*}(\Vh_{*}), n)) \cong K(F^pA^{*}(X; \Vh_{*}), n).
\eq
In terms of homotopy classes of maps, isomorphism \eqref{3.5pi} just states the fact 
\[
\pi_0K(F^pA^{*}(X; \Vh_{*}), n) \cong F^pH^n(X; \Vh_{*}). 
\]

Now let $E$ be a rationally even topological $\Omega$-spectrum together with the choice of a $p$-twisted fundamental cocycle $\iota$. Let $\Vh_{*}$ again denote $\pi_{*}E\otimes \C$.  
By Proposition \ref{htpypbprop}, we can calculate the homotopy pullback of \eqref{Efirstdef} objectwise. This implies that the space $E_n(p)(X)$ is homotopy equivalent to the homotopy pullback of the following diagram of simplicial sets 
\bq\label{EfirstdefX}
\xymatrix{
 & \Sing E_n(X) \ar[d]^{\iota_n^*}\\
K(F^pA^*(X; \Vh_{*}), n) \ar[r] & \Sing |K(A^*(\Vh_{*}), n)|(X).}
\eq

By Remark \ref{Elogiso}, this implies that the logarithmic Hodge filtered $E$-cohomology group $\Elog^n(p)(X)$ is isomorphic to the group of connected components of the space $E_n(p)(X)$, i.e., 
\[
\Elog^n(p)(X) = \pi_0(E_n(p)(X)).
\]

We can now read off from diagram \eqref{EfirstdefX} the following characterization of elements of $\Elog^n(p)(X)$. 

\begin{prop}\label{EltsofElog}
For rationally even spectrum $E$ and a smooth projective variety, an element of $\Elog^n(p)(X)$ is given by a triple 
\bq\label{triple}
q \colon X \to E_{n}, ~ \omega \in F^pA^n(X; \Vh_{*})_{\cl}, ~ \xi \in A^{n-1}(X; \Vh_{*})
\eq
such that $d \xi = \iota_n^*q - \omega$, where $q$ is a continuous map, $d$ denotes the differential in $A^{*}(X; \Vh_{*})$, and $\omega$ is a closed form. 
\end{prop}

\begin{remark}\label{IremE}
In view of Remark \ref{Irem}, we can rewrite diagram \eqref{EfirstdefX} as follows. 
Let $I$ denote again the map given by integration of forms. 
Following the argument in Remark \ref{Irem}, we see that $E_n(p)(X)$ fits into the following homotopy cartesian square of simplicial sets 
\[
\xymatrix{
E_n(p)(X) \ar[r] \ar[d] & \Sing E_n (X) \ar[d]^{\iota_n^*} \\
K(F^pA^*(X; \Vh_{*}), n) \ar[r]_-{I} & Z^n(X \times \Delta^{\bullet}; \Vh_{*}).}
\]

Hence we can represent an element in $\Elog^n(p)(X)$ also as a triple 
\[
q \colon X \to E_{n}, ~ \omega \in F^pA^n(X; \Vh_{*})_{\cl}, ~ h \in C^{n-1}(X; \Vh_{*})
\]
such that $\delta h = \iota_n^*q - I(\omega)$, where $\delta$ denotes the differential in $C^{*}(X; \Vh_{*})$. 
\end{remark}

\begin{remark}\label{remarkDeligne2}
Recall from Remark \ref{remarkDeligne1} that for $E=H\Z$, i.e., $E_n=K(\Z, n)$, $K(\Z, n)(p)$ represents Deligne-Beilinson cohomology in $\hosPresmc$.  
Then Remark \ref{IremE} just rephrases the well-known fact (see e.g. \cite[\S 12.3.2]{voisin1} or \cite{ev}) that an element in $H_{\Dh}^n(X; \Z(p))$ can be represented (in the notation of \cite{voisin1}) by a triple $(a^n_{\Z}, b_{F}^n, c_{\C}^{n-1})$ where $a^n_{\Z}$ is an integral singular cochain of degree $n$, $b^n_F$ is a form in $F^pA^n(X)$, and $c_{\C}^{n-1}$ is a complex singular cochain of degree $n-1$ such that $\delta c_{\C}^{n-1} = a^n_{\Z} - b^n_{F}$.  
\end{remark}

%%%%%%%%%%%%%%%%%%%%%%%%%%%%%%%%%%%%%%%

\section{A generalized Abel-Jacobi invariant}

In this section, we will always assume that $X$ is a projective smooth complex variety. 
Let $p$ be a fixed integer. 
Let $MU$ be the Thom spectrum representing complex cobordism. 
Recall that the homotopy groups of $MU$ vanish in all odd and in all negative degrees. Moreover, for $j\ge 0$, $\pi_{2j}MU$ is a finitely generated free abelian group. To shorten the notation we will again write 
\[
\Vh_* := \pi_*MU\otimes_{\Z}\C.
\]
Let $\iota$ be a $p$-twisted fundamental cocycle of $MU$. The reader may find a detailed discussion of fundamental cocycles for $MU$ in \cite[\S 5]{hfcbordism}. Here we just recall that $\iota$ comes equipped with an isomorphism 
\bq\label{splitting}
MU^*(X)_{\C} := MU^*(X)\otimes_{\Z} \C \cong H^*(X; \Vh_*) =  \bigoplus_{j\ge 0} H^{*+2j}(X; \Vh_{2j}).
\eq

\subsection{Cobordism, Jacobians, and Hodge structures}

We first define the generalized Jacobian we mentioned in the introduction. By the construction of the space $MU_n(p)$ as a homotopy pullback and by using isomorphism \eqref{KofFpA*graded}, we deduce that the groups $\MUlog^n(p)(X)$ sit in a long exact sequence of the form  (see also \cite[\S 4.2]{hfcbordism})
\[
\begin{array}{rl}
\ldots \to H^{n-1}(X; \Vh_* ) & \to \MUlog^n(p)(X) \to \\
\to MU^n(X)\oplus F^pH^n(X; \Vh_* ) & \to H^n(X; \Vh_*) \to \ldots
\end{array}
\]  
For $n=2p$, we can split this long exact sequence into the following short exact sequence 
\[
0 \to J_{MU}^{2p-1}(X) \to \MUlog^{2p}(p)(X) \to \HdgMU^{2p}(X) \to 0
\]
which is the bottom row of diagram \eqref{omdiagram}. 
The group on the left hand side is defined as
\[
J_{MU}^{2p-1}(X):= MU^{2p-1}(X)_{\C}/(F^{p}H^{2p-1}(X; \Vh_*) + MU^{2p-1}(X)).
\]
The group $\HdgMU^{2p}(X)$ is defined as the subgroup of $MU^{2p}(X)$ that is given as the pullback
\begin{equation}\label{foot}
\xymatrix{
\HdgMU^{2p}(X) \ar[d] \ar[r] & MU^{2p}(X) \ar[d] \\ 
F^{p}H^{2p}(X; \Vh_*) \ar[r] & H^{2p}(X; \Vh_*).}
\end{equation}

The space $X(\C)$ has the homotopy type of a finite complex. This implies that each group $MU^{k}(X)$ is finitely generated over $\Z$. 
Hence we may consider $MU^{k}(X)$ as a Hodge structure with the following filtration on $MU^{k}(X)_{\C}$. Using isomorphism \eqref{splitting} we set 
\[
F^iMU^{k}(X)_{\C} := \bigoplus_{j\ge 0}F^{i+j}H^{k+2j}(X; \C)\otimes_{\Z} \pi_{2j}MU.
\]

We can then interpret $J_{MU}^{2p-1}(X)$ as the Jacobian associated to the Hodge structure of weight $2p-1$ on $MU^{2p-1}(X)$:
\begin{equation}\label{Jnew}
J_{MU}^{2p-1}(X) = MU^{2p-1}(X)\otimes_{\Z}\C / (F^pMU^{2p-1}(X)_{\C} \oplus MU^{2p-1}(X)).
\end{equation}

Moreover, the canonical map $MU \to H\Z$ induces a map of Hodge structures
\[
(MU^{2p-1}(X), F^*MU^{2p-1}(X)_{\C}) \to (H^{2p-1}(X;\Z), F^*H^{2p-1}(X; \C))
\]
which induces the natural map 
\[
J_{MU}^{2p-1}(X) \to J^{2p-1}(X).
\]

%%%%%%%%%%%%%%%%%%%%%%%%%%%%%%%%%%%%%%

\subsection{A cycle map for algebraic cobordism}

Our goal in this subsection is to describe the vertical maps in diagram \eqref{omdiagram}. 
Let $\Omega^*(X)$ be the algebraic cobordism ring of $X$ of Levine and Morel. By \cite{lm}, $\Omega^*(-)$ is the universal oriented cohomology theory on $\Sm$. Moreover, in \cite[Theorem 7.10]{hfcbordism}, Michael J.\,Hopkins and the author showed that $\MUlog^{2*}(*)(-)$ is an oriented cohomology theory on $\Sm$, and hence the universality of $\Omega^*(-)$ induces a natural transformation 
\[
\tau \colon \Omega^*(-) \to \MUlog^{2*}(*)(-).
\]
Following \cite{lm}, for given $X \in \Sm$, $\tau$ is defined as follows. Recall that $\Omega^p(X)$ is generated by elements $[f\colon Y \to X]$ with $f$ a projective morphism in $\Sm$ of relative codimension $p$. Let $p_Y \colon Y\to \Spec \C$ be the structure map of $Y$. The class $[f]$ is equal to $f_*(p_Y^*(1_{\Omega}))$ where $f_*$ denotes the pushforward along $f$, $p_Y^*$ is the pullback along $p_Y$ in $\Omega^*(-)$, and $1_{\Omega}$ is the unit in $\Omega^0(\C)$.
The image of $[f \colon Y\to X]$ under $\tau \colon \Omega^p(X) \to \MUlog^{2p}(p)(X)$ is then defined as 
\[
\tau([f\colon Y \to X]) := f_*(p_Y^*(1_{\MUlog}))
\]
where now $f_*$ and $p_Y^*$ denote the pushforward and pullback in $\MUlog^{2*}(*)(-)$, respectively, and $1_{\MUlog}$ is the unit in $\MUlog^0(0)(\C)$. The natural map 
\[
\varphi_{\MUlog} \colon \Omega^*(X) \to \MUlog^{2*}(*)(X)
\]
is defined by $\varphi_{\MUlog} ([f]):= \tau([f])$. We remark that this is in fact a ring homomorphism for every $X\in \Sm$ (see \cite[Theorem 7.10]{hfcbordism}).

In order to further describe the element $\varphi_{\MUlog}([f])$ in $\MUlog^{2p}(p)(X)$, we will first look at the image of $[f]$ under the natural map  
\[
\varphi_{MU} \colon \Omega^p(X) \to MU^{2p}(X).
\]
We know that the image of $\varphi_{MU}$ lies in the subgroup $\HdgMU^{2p}(X)$. Therefore, we need a better understanding of the group $\HdgMU^{2p}(X)$. Using isomorphism \eqref{splitting} and the Hodge decomposition of complex cohomology, we can define subgroups of $MU^{k}(X)_{\C}$ 
\begin{equation}\label{decompdefn}
MU^{p,q}(X)_{\C} := \bigoplus_{j\ge 0}H^{p+j, q+j}(X; \C)\otimes_{\Z} \pi_{2j}MU.
\end{equation}
Then $MU^{k}(X)_{\C}$ splits into a direct sum 
\[
MU^{k}(X)_{\C} = \bigoplus_{p+q=k} MU^{p,q}(X)_{\C}.
\]

\begin{lemma}\label{MUpp}
Let $\gamma$ be an element of $\HdgMU^{2p}(X)$, and let $c \in H^{2p}(X; \Vh_*)$ be the image of $\gamma$ under isomorphism \eqref{splitting}. Then $c$ is given by a family of real cohomology classes $(c_j)_{j\ge 0}$ with 
\[
c_j \in H^{p+j, p+j}(X; \C) \otimes \pi_{2j}MU.
\]
\end{lemma}
\begin{proof}
This follows immediately from diagram \eqref{foot} and the fact that the image of $MU^{2p}(X)$ in $H^{2p}(X; \pi_*MU\otimes \C)$ factors through $H^{2p}(X; \pi_*MU\otimes \R)$. 
\end{proof}

\begin{remark}
In view of our notation \eqref{decompdefn}, we see that the group $\HdgMU^{2p}(X)$ can be identified with the elements in $MU^{2p}(X)$ whose image in $MU^{2p}(X)_{\C}$ lies in the subgroup $MU^{p,p}(X)_{\C}$. 
\end{remark}

\begin{lemma}\label{representingform}
With the notation of the previous lemma, the class $c$ can be represented by a family of closed forms $\omega =(\omega_j)$ in $F^pA^{2p}(X; \Vh_*)_{\cl}$ such that each $\omega_j$ is a real form of type $(p+j, p+j)$. 
\end{lemma}
\begin{proof}
This follows from Lemma \ref{MUpp}, Hodge symmetry and the uniqueness of the Hodge decomposition of complex cohomology. 
\end{proof}

\begin{remark}
Note that in both families $(c_j)$ and $(\omega_j)$ there are only finitely many nonzero elements. This is due to the fact that $X_{\an}$ is a compact complex manifold.
\end{remark}

This allows us to describe the image of $\varphi_{MU}$ as follows.

\begin{prop}\label{imageofphiMU}
Let $[f \colon Y \to X]$ be a generator in $\Omega^p(X)$. The image of $\varphi_{MU}(f)$ under isomorphism \eqref{splitting} is represented by a family of closed forms $\omega=(\omega_j)$ in $F^pA^{2p}(X; \Vh_*)_{\cl}$ such that each $\omega_j$ is a real form of type $(p+j, p+j)$. 
\end{prop}

As a consequence, we can also say more about the image of  
\[
\varphi_{\MUlog} \colon \Omega^p(X) \to \MUlog^{2p}(p)(X).
\]

\begin{prop}\label{imageofphiMUlog}
Let $[f \colon Y \to X]$ be a generator in $\Omega^p(X)$. The class $\varphi_{\MUlog}(f)$ can be represented by a triple 
\bq\label{tripleMUlog}
f_{\an} \colon Y_{\an} \to X_{\an} , ~ \omega \in F^pA^{2p}(X; \Vh_*)_{\cl}, ~ \xi \in A^{2p-1}(X; \Vh_*)
\eq
such that $d \xi = \iota^*(f_{\an}) - \omega$ and the image of $f_{\an}$ in $A^{2p}(X; \Vh_*)$ is a family of real forms of type $(p+j, p+j)$.
\end{prop}
\begin{proof}
Via the Pontryagin-Thom construction, the image $[f_{\an}]$ of $f$ in $MU^{2p}(X)$ corresponds to a continuous map $q \colon X \to MU_{2p}$. The assertion then follows from Propositions \ref{EltsofElog} and \ref{imageofphiMU}. 
\end{proof}

\begin{remark}
The proposition shows to what extend the class of $Y \to X$ in $\MUlog^{2p}(p)(X)$ contains more information than the corresponding images in $MU^{2p}(X)$ and $F^pH^{2p}(X; \Vh_*)$: $\varphi_{\MUlog}(f)$ remembers the homotopy that connects the images in $MU^{2p}(X)_{\C}$. We are going to exploit this fact in the construction of the Abel-Jacobi map in the next section. 
\end{remark}

%%%%%%%%%%%%%%%%

\subsection{The generalized Abel-Jacobi map}

Our final goal is to describe the Abel-Jacobi map 
\[
\Phi_{MU} \colon \Omtop^p(X) \to J_{MU}^{2p-1}(X)
\]
in diagram \eqref{omdiagram}.

Let $\alpha = [Y \to X]$ be a generator in $\Omega^p(X)$. By Remark \ref{IremE}, the image of $\alpha$ in $\MUlog^{2p}(p)(X)$ is given by a triple 
\bq\label{tripleMU}
q \colon X \to MU_{2p}, ~ \omega \in F^pA^{2p}(X; \Vh_{*})_{\cl}, ~ c \in C^{2p-1}(X; \Vh_{*})
\eq
such that $\delta c = I(\omega) - \iota_n^*q$, where $\delta$ denotes the differential in $C^{*}(X; \Vh_{*})$, $q$ is obtained via the Pontryagin-Thom construction and represents the class of $Y$ in $MU^{2p}(X)$, and $\omega$ represents the image of $[Y]$ in $H^{2p}(X; \Vh_*)$.

Now we assume that the image of $\alpha$ in $MU^{2p}(X)$ vanishes. This implies that both the cohomology class of $\omega$ and the homotopy class of $q$ are trivial. Hence there is a form $\eta \in F^pA^{2p-1}(X; \Vh_*)$ such that $d\eta = \omega$, and there is a homotopy $H$ from $q$ to the constant map which sends all of $X$ to the base point of $MU_{2p}$. We consider this homotopy as a map $H \colon X \to \Ph MU_{2p}$ from $X$ to the path space $\Ph MU_{2p}$ of $MU_{2p}$. 
After composition with $\iota \colon MU_{2p} \to K(\Vh_*, 2p)$, $H$ defines a map 
\[
\iota^*H \colon X \to \Ph K(\Vh_*, 2p).
\] 
This map, in turn, defines a cochain $\iota^*H$ in $C^{2p-1}(X; \Vh_*)$ which we, by abuse of notation, also denote by $\iota^*H$. By construction, the boundary of this cochain is $\delta(\iota^*H) = \iota^*q$, where we also write $\iota^*q$ for the cocycle corresponding to the map $\iota^*q \colon X \to K(\Vh_*, 2p)$. 

Hence the triple $(q, \omega, c)$ is homotopic to the triple $(0, 0, \iota^*H - I(\eta) + c)$, or, more precisely, they are path connected in the space $MU_{2p}(p)(X)$ and define the same element in $\pi_0(MU_{2p}(p)(X))$.  
Thus we can assume that $\varphi_{\MUlog}(\alpha)$ is represented by the triple $(0, 0, h)$ with 
\bq\label{hdef}
h := \iota^*H - I(\eta) + c \in C^{2p-1}(X; \Vh_*)~\text{and}~\delta h = 0. 
\eq
In particular, $h$ is a cocycle in $C^{2p-1}(X; \Vh_*)$ and represents an element $[h]$ in $H^{2p-1}(X; \Vh_*)\cong MU^{2p-1}(X)_{\C}$.

In constructing $[h]$, we chose a form $\eta$ and a homotopy $H$. The class of $[h]$ depends on these choices in the following way. First, let $\eta'$ be another form in $F^pA^{2p-1}(X; \Vh_*)$ such that $d\eta' = \omega$. The difference $\eta - \eta'$ is then a closed form. Hence $I(\eta - \eta')$ is a cocycle and defines a class in $F^pH^{2p-1}(X; \Vh_*)$. 

Second, let $H'$ be another homotopy from $q$ to the constant map. Considering $H$ and $H'$ as maps $X \to \Ph MU_{2p}$, we can compose $H$ and $H'$ as paths by first walking with doubled speed along $H$ and then with doubled speed in reversed direction of $H'$. Hence $H$ and $H'$ define a map $H*(-H') \colon X \to \Omega(MU_{2p})$ from $X$ to the loop space of $MU_{2p}$. Since $MU$ is an $\Omega$-spectrum, $\Omega(MU_{2p})$ is homeomorphic to $MU_{2p-1}$. Hence $H$ and $H'$ define a map $H*(-H') \colon X \to MU_{2p-1}$. After taking homotopy classes, we obtain an element $[H*(-H')] \in MU^{2p-1}(X)$. 

If we set $h= \iota^*H - I(\eta) + c$ and $h'=\iota^*H' - I(\eta') + c$, we obtain that the difference of $[h]$ and $[h']$ in $H^{2p-1}(X; \Vh_*)$ is the class of $\iota^*[H*(-H')] - I(\eta - \eta')$, where we also write $\iota$ for the map $MU_{2p-1} \to K(\Vh_*, 2p-1)$. Hence $[h] \in H^{2p-1}(X; \Vh_*)$ is well-defined modulo the images of the subgroups $MU^{2p-1}(X)$ and $F^pH^{2p-1}(X; \Vh_*)$. Hence we get the following theorem as an immediate consequence.

\begin{theorem}\label{mainthm}
Let $\alpha$ be an element in $\Omega^p(X)$ and $h$ be the cocycle defined in \eqref{hdef}. Then the image of $[h]$ in the quotient 
\[
J^{2p-1}_{MU}(X) = MU^{2p-1}(X)_{\C}/(F^{p}H^{2p-1}(X; \Vh_*) + MU^{2p-1}(X))
\]
is the image of $\alpha$ under $\Phi_{MU}$. 
\end{theorem}

\begin{remark}
We can interpret the class of $h$ also in the following way. Recall that, $X$ being a K\"ahler manifold, the natural map $H^{2p-1}(X; \R) \to H^{2p-1}(X; \C)/F^pH^{2p-1}(X; \C)$ is an isomorphism of $\R$-vector spaces. Because of our convention on the Hodge filtration and gradings, this implies that the natural map 
\[
H^{2p-1}(X; \pi_*MU\otimes \R) \to H^{2p-1}(X; \pi_*MU\otimes \C)/F^pH^{2p-1}(X; \pi_*MU\otimes \C)
\]
is an isomorphism of $\R$-vector spaces, too. 
For $h$ as in \eqref{hdef}, this implies that we can assume that $h$ defines a real class in $H^{2p-1}(X; \pi_*MU\otimes \R) \cong MU^{2p-1}(X)_{\R}$ which depends only on the image of $MU^{2p-1}(X)$. In other words, $h$ defines a unique element in 
\[
MU^{2p-1}(X)_{\R}/MU^{2p-1}(X) \cong MU^{2p-1}(X)\otimes \R/\Z
\]
which equals the image of $\alpha$ under the isomorphism 
\[
J_{MU}^{2p-1}(X) \cong MU^{2p-1}(X)\otimes \R/\Z.
\]
\end{remark}

\begin{remark}
We see that $\Phi_{MU}(\alpha)$ depends on {\it how} $\alpha$ vanishes in $MU^{2p}(X)$ and $F^pH^{2p}(X; \Vh_*)$. Hence calculating $\Phi_{MU}(\alpha)$ is in general a difficult task. But this is not different from the classical Abel-Jacobi map. For example, the calculations for Griffiths' famous examples are not done by just evaluating an integral of a form over some given chain, but use use more sophisticated arguments (see \cite{griffiths} or \cite{voisin2}). 
\end{remark}

\begin{remark}
Using known examples of non-trivial elements in the Griffiths group, we can produce examples of elements in $\Omtop^p(X)$ which lie in the kernel of $\theta \colon \Omega^p(X) \to CH^p(X)$ (see \cite[\S 7.3]{hfcbordism}). This shows that the new Abel-Jacobi invariant is able to detect certain elements in $\Omtop^*(X)$ which the classical invariant $\Phi$ would not see. Nevertheless, we do not yet have an example of the following type: an element in $\Omtop^*(X)$ which maps to the kernel of $\Phi$, but is neither in the kernel of $\theta$ nor in the kernel of $\Phi_{MU}$. This is a much more difficult task which requires a better understanding of the kernel of the map 
\[
MU^{2p-1}(X)\otimes \R/\Z \to H^{2p-1}(X; \R/\Z).
\]
\end{remark}

%
%%%%%%%%%%%%%%%%%%%%%%%%%%%%%%%%%%%%%%%%%
%
%%%%%%%%%%%%%%%%%%%%%%%%%%%%%%%%%%%%%%%%%
%
\bibliographystyle{amsalpha}

\end{document}